\documentclass[final,3p]{elsarticle}
\usepackage[latin1]{inputenc}
 \usepackage{graphics}
 \usepackage{graphicx}
 \usepackage{epsfig}
\usepackage{amssymb}
 \usepackage{amsthm}
 \usepackage{lineno}
 \usepackage{amsmath}
   \numberwithin{equation}{section}
\usepackage{mathrsfs}
\usepackage{color}

\journal{``Annals of Forest Research"} 
\newtheorem{thm}{Theorem}[section]
\newtheorem{cor}[thm]{Corollary}
\newtheorem{lem}[thm]{Lemma}
\newtheorem{prop}[thm]{Proposition}
\newtheorem{defn}[thm]{Definition}

\biboptions{sort&compress,square}
\allowdisplaybreaks
\begin{document}
\begin{frontmatter}
\author{Tong Wu$^{a}$}
\ead{wut977@nenu.edu.cn}
\author{Yong Wang$^{b,*}$}
\ead{wangy581@nenu.edu.cn}
\author{Xue Wang$^{b}$}
\ead{1203907183@qq.com}
\cortext[cor]{Corresponding author.}
\address{$^a$Department of Mathematics, Northeastern University, Shenyang, 110819, China}
\address{$^b$School of Mathematics and Statistics, Northeast Normal University,
Changchun, 130024, China}

\title{ Super twisted products}
\begin{abstract}
In the paper,we define the $W_2$-curvature tensor on super Riemannian manifolds. And we compute the curvature tensor, the Ricci tensor and the $W_2$-curvature tensor on super twisted product spaces. Furthermore, we investigate the $W_2$-curvature flat super twisted product manifolds. Finally, we get a result that a mixed Ricci-flat super twisted product semi-Riemannian manifold can
be expressed as a super warped product semi-Riemannian manifold.
\end{abstract}
\end{frontmatter}
\section{Introduction}
 \indent The concept of warped products was first introduced by
    Bishop and ONeil (see \cite{BO}) to construct examples of Riemannian
    manifolds with negative curvature. Singly warped products have a natural generalization. The (singly) twisted product $B\times_hF$ of two pseudo-Riemannian manifolds $(B,g_B)$ and
    $(F,g_F)$ with a smooth function $h:B\times F\rightarrow (0,\infty)$ is the product
    manifold $B\times F$ with the metric tensor $g=g_B\oplus
    h^2g_F.$ Here, $(B,g_B)$ is called the base manifold,
    $(F,g_F)$ is called as the fiber manifold and $h$ is called as
    the warping function. In Riemannian geometry,
    warped product manifolds and their generic forms have been used
    to construct new examples with interesting curvature properties
    since then. In \cite{DD}, F. Dobarro and E. Dozo had studied from the viewpoint of partial differential equations and variational methods,
    the problem of showing when a Riemannian metric of constant scalar curvature can be produced on a product manifolds by a warped product
    construction.
    In \cite{EJK}, Ehrlich, Jung and Kim got explicit solutions to warping function to have a constant scalar curvature for generalized
    Robertson-Walker space-times.
    In \cite{ARS}, explicit solutions were also obtained for the warping
    function to make the space-time as Einstein when the fiber is
    also Einstein. It is shown that a mixed Ricci-flat twisted product semi-Riemannian manifold can
be expressed as a warped product semi-Riemannian manifold in \cite{MF}.\\
      \indent  Pokhariyal and Mishra first defined the $W_2$-curvature tensor and they studied
its physical and geometrical properties in \cite{AC1}. In \cite{SO1} and \cite{SO}, Sular and \"{O}zgur studied warped product manifolds with a
         semi-symmetric metric connection and a
         semi-symmetric non-metric connection, they computed curvature of
         semi-symmetric metric connection and semi-symmetric non-metric connection
          and considered Einstein warped product manifolds with a
         semi-symmetric metric connection and a semi-symmetric non-metric connection. In \cite{W1}, Wang studied the Einstein multiply warped products with a semi-symmetric metric connection and the multiply warped products with a semi-symmetric metric connection with constant scalar curvature.\\
  \indent On the other hand, in \cite{BG}, the definition of super warped product spaces was given. Einstein warped products were studied in \cite{Ge}. In \cite{GDMVR}, several new super warped product spaces were given and the authors also studied the Einstein
  equations with cosmological constant in these new super warped product spaces. In \cite{WY}, Wang studied super warped product spaces with a semi-symmetric metric connection. In \cite{ss}, Shenawy. S and $\ddot{U}$nal. B studyed the $W_2$-curvature tensor on (singly) warped product
manifolds as well as on generalized Robertson-Walker and standard static space-time and investigated $W_2$-curvature flat warped product manifolds. In this paper, we define the $W_2$-curvature tensor on super twisted products. Our motivation is to study super twisted products and explore the Ricci tensor and the $W_2$-curvature tensor on
super twisted product manifolds.\\
\indent This paper is organized as follows. In Section \ref{Section:2}, we state some definitions of super manifolds and super Riemannian metrics. We also define the $W_2$-curvature tensor on super Riemannian manifolds. In Section \ref{Section:3}, we compute the curvature tensor, the Ricci tensor on super twisted product spaces. Further, we give the $W_2$-curvature tensor of the Levi-civita
 connection on super twisted product spaces. In Section \ref{Section:4}, we investigate $W_2$-curvature flat super twisted product manifolds. Finally, we get a result that a mixed Ricci-flat super twisted product semi-Riemannian manifold can
be expressed as a super warped product semi-Riemannian manifold.
\section{Preliminaries}
\label{Section:2}
In this section, we give some definitions about Riemannian supergeometry.
\begin{defn}\label{def1}(Definition 1 in \cite{BG}) A locally $\mathbb{Z}_2$-ringed space is a pair $S:= (|S|, \mathcal{O}_S)$ where $|S|$ is a second-countable
Hausdorff space, and a $\mathcal{O}_S$ is a sheaf of $\mathbb{Z}_2$-graded $\mathbb{Z}_2$-commutative associative unital $\mathbb{R}$-algebras, such that the
stalks $\mathcal{O}_{S,p}$, $p\in |S|$ are local rings.
\end{defn}
 \indent In this context, $\mathbb{Z}_2$-commutative means that any two sections $s,t\in \mathcal{O}_S(|U|),~~|U|\subset|S|$ open,
 of homogeneous
degree $|s|\in \mathbb{Z}_2$
and $|t|\in \mathbb{Z}_2$
commute up to the sign rule
$st=(-1)^{|s||t|}ts$.
 $\mathbb{Z}_2$-ring
space $U^{m|n}:= (U,C^{\infty}_{U^m}\otimes \wedge \mathbb{R}^n)$, is called standard
superdomain where $C^{\infty}_{U^m}$ is the sheaf of smooth functions on $U$ and $\wedge\mathbb{R}^n$ is
the exterior algebra of $\mathbb{R}^n$. We can employ (natural) coordinates $x^I:=(x^a,\xi^A)$ on any $\mathbb{Z}_2$-domain, where $x^a$ form a coordinate system on $U$ and the $\xi^A$
are formal coordinates.
\begin{defn}\label{def2}(Notation and preliminary concepts in \cite{BaJ})
 A supermanifold of dimension $m|n$ is a super ringed space
$M=(|M|, \mathcal{O}_M )$ that is locally isomorphic to $\mathbb{R}^{m|n}$ and $|M|$ is a second countable
and Hausdorff topological space.
\end{defn}
The tangent sheaf $\mathcal{T}M$ of a $\mathbb{Z}_2$-manifold $M$ is defined as the sheaf of derivations of sections of the structure
sheaf, i.e., $\mathcal{T}M(|U|) := {\rm Der}(\mathcal{O}_M(|U|)),$ for arbitrary open set $|U|\subset |M|.$ Naturally, this is a sheaf of locally free $\mathcal{O}_M$-modules. Global sections of the tangent sheaf are referred to as {\it vector fields}. We denote the $\mathcal{O}_M (|M|)$-module
of vector fields as ${\rm Vect}(M)$. The dual of the tangent sheaf is the {\it cotangent sheaf}, which we denote as $\mathcal{T}^*M$.
This is also a sheaf of locally free $\mathcal{O}_M$-modules. Global section of the cotangent sheaf we will refer to as {\it one-forms}
and we denote the $\mathcal{O}_M(|M|)$-module of one-forms as $\Omega^1(M)$.
\begin{defn}\label{def3}(Definition 4 in \cite{BG})
 A Riemannian metric on a $\mathbb{Z}_2$-manifold M is a $\mathbb{Z}_2$-homogeneous, $\mathbb{Z}_2$-symmetric, non-degenerate,
$\mathcal{O}_M$-linear morphisms of sheaves $\left<-,-\right>_g:~~\mathcal{T}M\otimes \mathcal{T}M\rightarrow \mathcal{O}_M.$
A $\mathbb{Z}_2$-manifold equipped with a Riemannian metric is referred to as a Riemannian $\mathbb{Z}_2$-manifold.
\end{defn}
We will insist that the Riemannian metric is homogeneous with respect to the $\mathbb{Z}_2$-degree, and we will denote
the degree of the metric as $|g| \in \mathbb{Z}_2$.
Explicitly, a Riemannian metric has the following properties:\\
(1)$ |\left<X,Y\right>_g |= |X| + |Y |+ |g|,$\\
(2)$\left<X,Y\right>_g =(-1)^{|X||Y|}\left<Y,X\right>_g,$\\
(3) If $\left<X,Y\right>_g = 0$ for all $Y \in Vect(M),$ then $X = 0,$\\
(4) $\left<fX+Y,Z\right>_g =f\left<X,Z\right>_g +\left<Y,Z\right>_g ,$\\
for arbitrary (homogeneous) $ X, Y, Z \in {\rm Vect}(M)$ and $f \in C^{\infty}(M)$. We will say that a Riemannian metric is
even if and only if it has degree zero. Similarly, we will say that a Riemannian metric is odd if and only
if it has degree one. Any Riemannian metric we consider will be either even or odd as we will only be
considering homogeneous metrics.\\
\begin{defn}\label{fff4}(Definition 9 in \cite{BG}) An affine connection on a $\mathbb{Z}_2$-manifold is a $\mathbb{Z}_2$-degree preserving map\\
$$\nabla:{\rm Vect}(M)\times {\rm Vect}(M)\rightarrow {\rm Vect}(M);~~(X,Y)\mapsto \nabla_XY,$$
which satisfies the followings\\
1) Bi-linearity $$\nabla_X(Y+Z)=\nabla_XY+\nabla_XZ;~~\nabla_{X+Y}Z=\nabla_XZ+\nabla_YZ,$$
2)$C^{\infty}(M)$-linearrity in the first argument
$$\nabla_{fX}Y=f\nabla_XY,$$
3)The Leibniz rule
$$\nabla_X(fY)=X(f)Y+(-1)^{|X||f|}f\nabla_XY,$$
for all homogeneous $X,Y,Z\in {\rm Vect}(M)$ and $f\in C^{\infty}(M)$.
\end{defn}
\begin{defn}\label{def4}(Definition 10 in \cite{BG})
 The torsion tensor of an affine connection \\
 $T_\nabla:~~{\rm Vect}(M)\otimes_{C^{\infty}(M)} {\rm Vect}(M)\rightarrow {\rm Vect}(M)$
 is
defined as
$$T_\nabla(X,Y):=\nabla_XY-(-1)^{|X||Y|}\nabla_YX-[X,Y],$$
for any (homogeneous) $X, Y \in {\rm Vect}(M)$. An affine connection is said to be symmetric if the torsion vanishes.
\end{defn}
\begin{defn}\label{def5}(Definition 11 in \cite{BG})
An affine connection on a Riemannian $\mathbb{Z}_2$-manifold $(M, g)$ is said to be metric compatible if
and only if
$$ X\left<Y,Z\right>_g=\left<\nabla_XY,Z\right>_g+(-1)^{|X||Y|}\left<Y,\nabla_XZ\right>_g,$$
for any $X, Y, Z\in  {\rm Vect}(M)$.
\end{defn}
\begin{thm}\label{thm1}(Theorem 1 in \cite{BG}) There is a unique symmetric (torsionless) and metric compatible
affine connection $\nabla^L$ on a Riemannian $\mathbb{Z}_2$-manifold $(M, g)$ which satisfies the Koszul formula
\begin{align}\label{a1}
2\left<\nabla^L_XY,Z\right>_g&=X\left<Y,Z\right>_g+\left<[X,Y],Z\right>_g\nonumber\\
&+(-1)^{|X|(|Y|+|Z|)}(Y\left<Z,X\right>_g-\left<[Y,Z],X\right>_g)\nonumber\\
&-(-1)^{|Z|(|X|+|Y|)}(Z\left<X,Y\right>_g-\left<[Z,X],Y\right>_g),
\end{align}
for all homogeneous $X,Y,Z\in {\rm Vect}(M)$.
\end{thm}
\begin{defn}\label{def6}(Definition 13 in \cite{BG})
The Riemannian curvature tensor of an affine connection
$$R_\nabla:~~{\rm Vect}(M)\otimes_{C^{\infty}(M)} {\rm Vect}(M)\otimes_{C^{\infty}(M)} {\rm Vect}(M)\rightarrow {\rm Vect}(M)$$
is defined as
\begin{align}\label{a2}
R_\nabla(X, Y )Z =\nabla_X\nabla_Y-(-1)^{|X||Y|}\nabla_Y\nabla_X-\nabla_{[X,Y]}Z,
\end{align}
for all $X, Y$ and $Z \in {\rm Vect}(M)$.
\end{defn}
Directly from the definition it is clear that
\begin{align}\label{a3}
R_\nabla(X, Y )Z =-(-1)^{|X||Y|}R_\nabla(Y,X)Z,
\end{align}
for all $X, Y$ and $Z \in {\rm Vect}(M)$.
\begin{defn}\label{def7}(Definition 14 in \cite{BG})
 The Ricci curvature tensor of an affine connection is the symmetric rank-$2$ covariant tensor
defined as
 \begin{equation}\label{a4}
Ric_\nabla(X, Y ):=(-1)^{|\partial_{x^I}|(|\partial_{x^I}|+|X|+|Y|)}\frac{1}{2}\left[R_\nabla(\partial_{x^I},X)Y+(-1)^{|X||Y|}R_\nabla(\partial_{x^I},Y)X\right]^I,
\end{equation}
where $X,Y\in {\rm Vect}(M)$ and $[~~]^I$ denotes the coefficient of $\partial_{x^I}$ and $\partial_{x^I}$ is the natural frame of $\mathcal{T}M$.
\end{defn}
\begin{defn}\label{def8}(Definition 16 in \cite{BG})
Let $f \in C^{\infty}(M)$ be an arbitrary function on a Riemannian $\mathbb{Z}_2$-manifold $(M, g)$. The gradient
of $f$ is the unique vector field ${\rm grad}_gf$ such that
 \begin{equation}\label{a5}
X(f)=(-1)^{|f||g|}\left<X,{\rm grad}_gf\right>_g,
\end{equation}
for all $X \in {\rm Vect}(M)$.
\end{defn}
\begin{defn}\label{def9}(Definition 17 in \cite{BG})
Let $(M, g)$ be a Riemannian $\mathbb{Z}_2$-manifold
and let $\nabla^L$ be the associated Levi-Civita connection.
The covariant divergence is the map ${\rm Div}_L:{\rm Vect}(M)\rightarrow C^{\infty}(M)$, given by
\begin{equation}\label{a6}
{\rm Div}_L(X)=(-1)^{|\partial_{x^I}|(|\partial_{x^I}|+|X|)}(\nabla_{\partial_{x^I}}X)^I,
\end{equation}
for any arbitrary $X \in {\rm Vect}(M)$.
\end{defn}
\begin{defn}\label{def10}(Definition 18 in \cite{BG}) Let $(M, g)$ be a Riemannian $\mathbb{Z}_2$-manifold
and let $\nabla^L$ be the associated Levi-Civita connection.
The connection Laplacian (acting on functions) is the differential operator of  $\mathbb{Z}_2$-degree $|g|$ defined as
\begin{equation}\label{a7}
\triangle_g(f)={\rm Div}_L({\rm grad}_gf),
\end{equation}
for any and all $f \in C^{\infty}(M).$
\end{defn}
\begin{defn}\label{def11}Let $(M, g)$ be a Riemannian $\mathbb{Z}_2$-manifold and $\nabla$ be the Levi-Civita connection associated to the Riemannian metric $g^M$. Let $D\subseteq TM$ be a super distribution and $D^\bot\subseteq
TM$ is the orthogonal distribution to $D$, then $g^M=g^D+g^{D^\bot}$. Let $\pi^D:TM\rightarrow D$, $\pi^{D^\bot}:TM\rightarrow D^\bot$ be the projections.
For $X,Y \in\Gamma(D)$, we define $\nabla^D_XY=\pi^D(\nabla_XY)$,
then we have the fundamental form of submanifold on Riemannian $\mathbb{Z}_2$-manifold $(M, g)$
\begin{equation}\label{a8}
\nabla_XY=\nabla^D_{X}Y+B(X,Y),
\end{equation}
\begin{equation}\label{ao8}
\nabla_X\xi=-(-1)^{|X||\xi|}A_\xi X+L^{\bot}_X\xi,\\
\end{equation}
 where $B(X,Y)=\pi^{D^\bot}\nabla_XY$, $L^{\bot}_X\xi=\pi^{D^\bot}\nabla_X\xi$ and $\pi^D\nabla_X\xi=-(-1)^{|X||\xi|}A_\xi X$ for any homogenous $\xi\in\Gamma(D^{\bot})$ and $X, Y\in\Gamma(D)$.
\end{defn}
Then
\begin{align}\label{a1ll}
&B(fX,Y)=fB(X,Y),~~B(X,fY)=(-1)^{|f||X|}B(X,Y),~~B(X,Y)=(-1)^{|X||Y|}B(Y,X)+\pi^{D^\bot}[X,Y];\nonumber\\
&A_{f\xi}X=fA_{\xi}X,~~A_{\xi}fX=(-1)^{|f||\xi|}A_{\xi}X,~~g^{D^\bot}(B(X,Y),\xi)=(-1)^{|X|(|Y|+|\xi|)}g^D(Y,A_{\xi}X).
\end{align}
When $D$ is a submanifold of $M$, we also have similar formula.\\
\begin{defn}\label{def1p}Let $(M^{m,n}, g)$ be a Riemannian $\mathbb{Z}_2$-manifold, the $W_2$-curvature tensor is also given by
\begin{align}\label{a8}
W_2(X,Y,Z,T) = g(K(X,Y)T,Z),
\end{align}
for any homogenous $X,Y,Z,T\in {\rm Vect}(M)$ and where
\begin{align}\label{a448}
K(X,Y)T:=R(X,Y)T-\frac{1}{(m-n-1)}[X\cdot Ric(Y,T)-(-1)^{|Y||T|}Ric(X,T)Y].
\end{align}
\end{defn}

\section{The $W_2$-curvature tensor on super twisted products}
\label{Section:3}
Let $(M=M_1\times_\mu M_2,g_\mu=\pi^*_1 g_1+\pi^*_1(\mu)\pi_2^*g_2)$ be the super twisted product with $|g|=|g_1|=|g_2|=0$ and $|\mu|=0$ and $\mu\in C^\infty(M)$ and its body $\varepsilon(\mu)>0$. For simplicity, we assume that $\mu=h^2$ with $|h|=0$. Let $\nabla^{L,\mu}$ be the Levi-Civita connection on $(M,g_\mu)$ and $\nabla^{L,M_1}$ (resp. $\nabla^{L,M_2}$) be the Levi-Civita connection on $(M_1,g_1)$ (resp. $(M_2,g_2)$).
\begin{lem}\label{lem1}
For $X,Y,Z\in{\rm Vect}(M_1)$ and $U,W,V\in {\rm Vect}(M_2)$, we have
\begin{align}\label{b1}
&(1)\nabla^{L,\mu}_XY=\nabla^{L,M_1}_XY,\nonumber\\
&(2)\nabla^{L,\mu}_XU=\frac{X(h)}{h}U,\nonumber\\
&(3)\nabla^{L,\mu}_UX=(-1)^{|U||X|}\frac{X(h)}{h}U,\nonumber\\
&(4)\nabla^{L,\mu}_UW=\frac{U(h)}{h}W+(-1)^{|U||W|}\frac{W(h)}{h}U-(-1)^{|V|(|U|+|W|)}\frac{g_2(U,W)}{h}{\rm grad}_{g_2}h\nonumber\\
&-hg_2(U,W){\rm grad}_{g_1}h+\nabla^{L,M_2}_UW.
\end{align}
\end{lem}
\begin{proof}
(1)By (\ref{a1}) and $[X,V]=0$, we have $g_\mu(\nabla^{L,\mu}_XY,Z)=g_1(\nabla^{L,M_1}_XY,Z)$ and $g_\mu(\nabla^{L,\mu}_XY,V)=0$, so (1) holds.\\
(2)Similarly, we have $g_\mu(\nabla^{L,\mu}_XU,Y)=0$ and $2g_\mu(\nabla^{L,\mu}_XU,V)=\frac{X(\mu)}{\mu}g_\mu(U,V)$, so (2) holds by $\mu=h^2$.\\
(3)By (2) and $\nabla^{L,\mu}$ having no torsion, we have (3).\\
(4)By (\ref{a1}) and (\ref{a5}), we have
\begin{align}
2g_\mu(\nabla^{L,\mu}_UW,X)&=-(-1)^{|X|(|U|+|W|)}X(\mu)g_2(U,W)\nonumber\\
&=-(-1)^{|X|(|U|+|W|)}g_1(X,{\rm grad}_{g_1}(\mu))g_2(U,W)\nonumber\\
&=-g_\mu(g_2(U,W){\rm grad}_{g_1}(\mu),X).
\end{align}
And
\begin{align}\label{bbbb}
&2g_\mu(\nabla^{L,\mu}_UW,V)\nonumber\\
&=U(h^2)g_2(W,V)+(-1)^{|U||W|}W(h^2)g_2(U,V)-(-1)^{|V|(|U|+|W|)}V(h^2)g_2(U,W)+2g_\mu(\nabla^{L,M_2}_UW,V),
\end{align}
then we have
\begin{align}\label{bbpp}
P^{M_2}\nabla^{L,\mu}_UW&=\frac{U(h)}{h}W+(-1)^{|U||W|}\frac{W(h)}{h}U-(-1)^{|V|(|U|+|W|)}\frac{g_2(U,W)}{h}{\rm grad}_{g_2}h,
\end{align}
so (4) holds.
\end{proof}
Let $R^{L,\mu}$ denotes the curvature tensor of the Levi-Civita connection on $(M^{m,n},g_\mu)$ and let $R^{L,M_1}$ (resp. $R^{L,M_2}$) be the curvature tensor of the Levi-Civita connection on $(M_1,g_1)$ (resp. $(M_2,g_2)$). Let $H^h_{M_1}(X,Y):=XY(h)-\nabla^{L,{M_1}}_XY(h)$, then
$H^h_{M_1}(fX,Y)=fH^h_{M_1}(X,Y)$ and $H^h_{M_1}(X,fY)=(-1)^{|f||X|}fH^h_{M_1}(X,Y)$, where $H^h_{M_1}$ is a $(0,2)$ tensor.
\begin{prop}\label{prop2}
For $X,Y,Z\in{\rm Vect}(M_1)$ and $U,V,W\in {\rm Vect}(M_2)$, we have
\begin{align}\label{b3}
&(1)R^{L,\mu}(X,Y)Z=R^{L,M_1}(X,Y)Z,\nonumber\\
&(2)R^{L,\mu}(V,X)Y=-(-1)^{|V|(|X|+|Y|)}\frac{H^h_{M_1}(X,Y)}{h}V,\nonumber\\
&(3)R^{L,\mu}(X,Y)V=0,\nonumber\\
&(4)R^{L,\mu}(V,W)X=(-1)^{|W||X|}V\left(\frac{X(h)}{h}\right)W-(-1)^{|V|(|W|+|X|)}W\left(\frac{X(h)}{h}\right)V,\nonumber\\
&(5)R^{L,\mu}(X,V)W=(-1)^{(|X|+|V|)|W|}W\left(\frac{X(h)}{h}\right)V-(-1)^{|X|(|V|+|W|)+|V||W|}g_2(W,V){\rm grad}_{g_2}\frac{X(h)}{h}\nonumber\\
&-(-1)^{|X|(|V|+|W|)}\frac{g_\mu(V,W)}{h}\nabla^{L,M_1}_X({\rm grad}_{g_1}h),\nonumber\\
&(6)R^{L,\mu}(V,W)U=R^{L,M_2}(V,W)U+(-1)^{|U||W|}g_\mu(V,U){\rm grad}_{g_1}\frac{W(h)}{h}-(-1)^{(|U|+|W|)|V|}g_\mu(W,U){\rm grad}_{g_1}\frac{V(h)}{h}\nonumber\\
&-(-1)^{|V|(|W|+|U|)}\frac{({\rm grad}_{g_1}h)(h)}{h^2}g_2(W,U)V+(-1)^{|W||U|}\frac{({\rm grad}_{g_1}h)(h)}{h^2}g_2(V,U)W.
\end{align}
\end{prop}
\begin{proof}(1)By Lemma 3.1 and (\ref{a2}), we can get (1).\\
\indent (2)By Lemma 3.1 and the Leibniz rule, we have
\begin{align}
\nabla^{L,\mu}_V\nabla^{L,\mu}_XY=(-1)^{|V|(|X|+|Y|)}\frac{\nabla^{L,M_1}_XY(h)}{h}V,~~-\nabla^{L,\mu}_X\nabla^{L,\mu}_VY=-(-1)^{|V|(|X|+|Y|)}\frac{XY(h)}{h}V,
\end{align}
by (\ref{a2}) and $[V,X]=0$ and the definition of $H^h_{M_1}(X,Y)$, we get (2).\\
\indent (3) By Lemma 3.1, we have $\nabla^{L,\mu}_X\nabla^{L,\mu}_YV=\frac{XY(h)}{h}V$ and by (\ref{a2}) and the definition of $[X,Y]$, we get (3).\\
\indent(4)By Lemma 3.1, we have
\begin{align}\label{l1}
\nabla^{L,\mu}_V\nabla^{L,\mu}_WX=(-1)^{|W||X|}\left[V\left(\frac{X(h)}{h}\right)W+(-1)^{|V||X|}\frac{X(h)}{h}\nabla^{L,\mu}_VW\right],
\end{align}
\begin{align}\label{l2}
-\nabla^{L,\mu}_W\nabla^{L,\mu}_VX=-(-1)^{|V||X|}\left[W\left(\frac{X(h)}{h}\right)V+(-1)^{|W||X|}\frac{X(h)}{h}\nabla^{L,\mu}_WV\right],
\end{align}
\begin{align}\label{l3}
-\nabla^{L,\mu}_{[V,W]}X=(-1)^{(|V|+|W|)|X|}\frac{X(h)}{h}[V,W],
\end{align}
then by (\ref{a2}) and $\nabla^{L,\mu}$ having no torsion, we get (4).\\
\indent (5)For $W_1\in {\rm Vect}(M_2)$, by (4) and (4.12) in \cite{Go}, we have
\begin{align}\label{l4}
&g_\mu(R^{L,\mu}(X,V)W,W_1)\nonumber\\
&=(-1)^{(|X|+|V|)(|W|+|W_1|)}g_\mu(R^{L,\mu}(W,W_1)X,V)\nonumber\\
&=(-1)^{(|X|+|V|)(|W|+|W_1|)}g_\mu((-1)^{|X||W_1|}W\left(\frac{X(h)}{h}\right)W_1-(-1)^{(|X|+|W_1|)|W|}W_1\left(\frac{X(h)}{h}\right)W,V)\nonumber\\
&=(-1)^{(|X|+|V|)|W|+|V||W_1|}W\left(\frac{X(h)}{h}\right)g_\mu(W_1,V)-(-1)^{(|X|+|W|)|W_1|+|V|(|W|+|W_1|)}W_1\left(\frac{X(h)}{h}\right)g_\mu(W,V),
\end{align}
by (\ref{a5}), we have
\begin{align}\label{ik}
W_1\left(\frac{X(h)}{h}\right)=(-1)^{|W_1||X|}g_2({\rm grad}_{g_2}\frac{X(h)}{h},W_1),
\end{align}
then,
\begin{align}\label{mmm}
P^{M_2}R^{L,\mu}(X,V)W=(-1)^{|X|(|V|+|W|)}W\left(\frac{X(h)}{h}\right)V-(-1)^{|X|(|V|+|W|)+|V||W|}g_2(W,V){\rm grad}_{g_2}\frac{X(h)}{h}.\nonumber\\
\end{align}
By Proposition 9 in \cite{BG} and (\ref{a2}), we have
\begin{align}
g_\mu(R^{L,\mu}(X,V)W,Y)&=-(-1)^{|W||Y|}g_\mu(R^{L,\mu}(X,V)Y,W)
&=-(-1)^{(|W|+|V|)|Y|}\frac{H^h_{M_1}(X,Y)}{h}g_\mu(V,W).
\end{align}
By the definition of ${\rm grad}_{g_1}(h)$ and $\nabla^{L,M_1}$ preserving the metric, we can get
\begin{align}
g_1(\nabla^{L,M_1}_X({\rm grad}_{g_1}h),Y)=(-1)^{|Y||g_1|}H^h_{M_1}(X,Y)=H^h_{M_1}(X,Y),
\end{align}
so
\begin{align}
g_\mu(R^{L,\mu}(X,V)W,Y)=-(-1)^{|X|(|V|+|W|)}g_\mu(\frac{g_\mu(V,W)}{h}\nabla^{L,M_1}_X({\rm grad}_{g_1}h),Y).
\end{align}
By (\ref{l1}) and (\ref{l4}), we get (5).\\
\indent (6) By (\ref{a5}), we have
\begin{align}\label{bb0}
g_\mu(R^{L,\mu}(V,W)U,X)&=-(-1)^{|X||U|}g_\mu(R^{L,\mu}(V,W)X,U)\nonumber\\
&=-(-1)^{|X||U|}g_\mu((-1)^{|X||W|}V\left(\frac{X(h)}{h}\right)W-(-1)^{|V|(|X|+|W|)}W\left(\frac{X(h)}{h}\right)V,U)\nonumber\\
&=(-1)^{|X||U|+|V|(|X|+|W|)}g_\mu(W\left(\frac{X(h)}{h}\right)V,U)-(-1)^{|X|(|U|+|W|)}g_\mu(V\left(\frac{X(h)}{h}\right)W,U).
\end{align}
By $XW=(-1)^{|X||W|}WX$ and $XV=(-1)^{|X||V|}VX$, we have
\begin{align}\label{sss}
P^{M_1}R^{L,\mu}(V,W)U=(-1)^{|U||W|}g_\mu(V,U){\rm grad}_{g_2}\frac{W(h)}{h}-(-1)^{(|U|+|W|)|V|}g_\mu(W,U){\rm grad}_{g_2}\frac{V(h)}{h}.
\end{align}
By Definition \ref{def6} and Definition \ref{def11}, we have
\begin{equation}
R^D(X,Y)Z=(R(X,Y)Z)^\top-(-1)^{|Y||Z|}A_{B(X,Z)}Y+(-1)^{|X|(|Y|+|Z|)}A_{B(Y,Z)}X,\nonumber\\
\end{equation}
where $R^D$ (resp. $R$) denotes the curvature tensor of $\nabla^D$ (resp. $\nabla$) and $R^\top$ denotes tangential component of $R$.\\
Then by $g(B(X,Y),\xi)=(-1)^{(|Y|+|\xi|)|X|}g(Y,A_\xi X),$ we have
\begin{align}\label{jj3}
g(A_\xi X,Y)=(-1)^{|Y|(|X|+|\xi|)}g(Y,A_\xi X)=(-1)^{|\xi|(|X|+|Y|)}g(B(X,Y),\xi),
\end{align}
so
\begin{align}\label{nn}
&g(R^D(X,Y)Z,W)\nonumber\\
&=g(R(X,Y)Z),W)-(-1)^{|Y||Z|}g(A_{B(X,Z)}Y,W)+(-1)^{|X|(|Y|+|Z|)}g(A_{B(Y,Z)}X,W)\nonumber\\
&=g(R(X,Y)Z),W)-(-1)^{|X|(|Y|+|W|)+|W||Z|}g(B(Y,W),B(X,Z))+(-1)^{|W|(|Y|+|Z|)}g(B(X,W),B(Y,Z)).
\end{align}
By $B(U,W)=-hg_2(U,W){\rm grad}_{g_1}h=-\frac{g_\mu(U,W)}{h}{\rm grad}_{g_1}h,$ we have
\begin{align}\label{oio}
g_\mu(R^{M_2}(V,W)U,W_1)&=g_\mu(R^{L,\mu}(X,Y)Z),W_1)-(-1)^{|V|(|W|+|W_1|)+|W_1||U|}g_\mu(B(W,W_1),B(V,U))\nonumber\\
&+(-1)^{|W_1|(|W|+|U|)}g_\mu(B(V,W_1),B(W,U))\nonumber\\
&=g_\mu(R^{L,\mu}(X,Y)Z),W_1)+(-1)^{|V|(|W|+|U|)}\frac{|{\rm grad}_{g_1}h|^2_g}{h^2}g_\mu(W,U)g_\mu(V,W_1)\nonumber\\
&-(-1)^{|W||U|}\frac{|{\rm grad}_{g_1}h|^2_g}{h^2}g_\mu(V,U)g_\mu(W,W_1),\nonumber\\
&=g_\mu(R^{L,\mu}(X,Y)Z),W_1)+(-1)^{|V|(|W|+|U|)}\frac{({\rm grad}_{g_1}h)(h)}{h^2}g_\mu(W,U)g_\mu(V,W_1)\nonumber\\
&-(-1)^{|W||U|}\frac{({\rm grad}_{g_1}h)(h)}{h^2}g_\mu(V,U)g_\mu(W,W_1),
\end{align}
then, we get (6).
\end{proof}
In the following, we compute the Ricci tensor of manifold $M^{m,n}$. Let $M_1$ (resp. $M_2$) have the $(p,m_1)$ (resp. $(q,m_2)$) dimension, where $n_1=p-m_1,~n_2=q-m_2$ and $m-n=n_1+n_2$. Let $\partial_{x^I}=\{\partial_{x^a},\partial_{\xi^A}\}$ (resp.
$\partial_{y^J}=\{\partial_{y^b},\partial_{\eta^B}\}$) denote
the natural tangent frames on $M_1$ (resp. $M_2$). Let ${\rm Ric}^{L,\mu}$ (resp. ${\rm Ric}^{L,M_1}$, ${\rm Ric}^{L,M_2}$)  denote the Ricci tensor of $(M,g_\mu)$ (resp. $(M_1,g_1)$, $(M_2,g_2)$).  Then by (\ref{a4}), (\ref{a7}) and (\ref{b3}), we have
\begin{prop}\label{prop12}
The following equalities holds
\begin{align}\label{b18}
(1){\rm Ric}^{L,\mu}(\partial_{x^L},\partial_{x^K})&={\rm Ric}^{L,M_1}(\partial_{x^L},\partial_{x^K})-\frac{(q-m_2)}{h}H^h_{M_1}(\partial_{x^L},\partial_{x^K}),\nonumber\\
(2){\rm Ric}^{L,\mu}(\partial_{x^L},\partial_{y^J})&=-(q-m_2-1)(-1)^{|\partial_{x_L}||\partial_{y^J}|}\partial_{y^J}\left(\frac{\partial_{x^L}(h)}{h}\right),\nonumber\\
(3){\rm Ric}^{L,\mu}(\partial_{y^J},\partial_{x^L})&=-(q-m_2-1)\partial_{y^J}\left(\frac{\partial_{x^L}(h)}{h}\right),\nonumber\\
(4){\rm Ric}^{L,\mu}(\partial_{y^L},\partial_{y^J})&={\rm Ric}^{L,M_2}(\partial_{y^L},\partial_{y^J})-g_\mu(\partial_{y^L},\partial_{y^J})\cdot[\frac{\triangle^L_{g_1}(h)}{h}+(q-m_2-1)\frac{({\rm grad}_{g_1}h)(h)}{h^2}].
\end{align}
\end{prop}
\begin{proof}
(1)By Definition \ref{def7} and Proposition \ref{prop2}, we have
\begin{align}\label{f17}
{\rm Ric}^{L,\mu}(\partial_{x^L},\partial_{x^K})&=\sum_I(-1)^{|\partial_{x^I}|(|\partial_{x^I}|+|\partial_{x^L}|+|\partial_{x^K}|)}\frac{1}{2}[R^{L,\mu}(\partial_{x^I},\partial_{x^L})\partial_{x^K}+(-1)^{|\partial_{x^L}||\partial_{x^K}|}R^{L,\mu}(\partial_{x^I},\partial_{x^K})\partial_{x^L}]^I\nonumber\\
&+\sum_J(-1)^{|\partial_{y^J}|(|\partial_{y^J}|+|\partial_{x^L}|+|\partial_{x^K}|)}\frac{1}{2}[R^{L,\mu}(\partial_{y^J},\partial_{x^L})\partial_{x^K}+(-1)^{|\partial_{x^L}||\partial_{x^K}|}R^{L,\mu}(\partial_{y^J},\partial_{x^K})\partial_{x^L}]^J\nonumber\\
&={\rm Ric}^{L,M_1}(\partial_{x^L},\partial_{x^K})+\sum_J(-1)^{|\partial_{y^J}|(|\partial_{y^J}|+|\partial_{x^L}|+|\partial_{x^K}|)}\frac{1}{2}\bigg[-(-1)^{|\partial_{y^J}|(|\partial_{x^L}|+|\partial_{x^K}|)}\frac{H_{M_1}^h(\partial_{x^L},\partial_{x^K})}{h}\nonumber\\
&\partial_{y^J}-(-1)^{|\partial_{x^L}||\partial_{x^K}|}(-1)^{|\partial_{y^J}|(|\partial_{x^L}|+|\partial_{x^K}|)}\frac{H_{M_1}^h(\partial_{x^K},\partial_{x^L})}{h}\partial_{y^J}\bigg]^J\nonumber\\
&={\rm Ric}^{L,M_1}(\partial_{x^L},\partial_{x^K})-\sum_J(-1)^{|\partial_{y^J}||\partial_{y^J}|}\frac{1}{2}\bigg[\frac{H_{M_1}^h(\partial_{x^L},\partial_{x^K})}{h}+(-1)^{|\partial_{x^L}||\partial_{x^K}|}\frac{H_{M_1}^h(\partial_{x^K},\partial_{x^L})}{h}\bigg],
\end{align}
by $(-1)^{|\partial_{y^J}||\partial_{y^J}|}=\sum_{j=1}^q(-1)^0+\sum_{k=1}^{m_2}(-1)^1=q-m_2$, we have
\begin{align}\label{mnm}
{\rm Ric}^{L,\mu}(\partial_{x^L},\partial_{x^K})={\rm Ric}^{L,M_1}(\partial_{x^L},\partial_{x^K})-\frac{(q-m_2)}{2h}\bigg[\frac{H_{M_1}^h(\partial_{x^L},\partial_{x^K})}{h}+(-1)^{|\partial_{x^L}||\partial_{x^K}|}\frac{H_{M_1}^h(\partial_{x^K},\partial_{x^L})}{h}\bigg],\nonumber\\
\end{align}
by $\nabla^{L,M_1}$ having no torsion, we have
\begin{align}\label{lll}
&\frac{H_{M_1}^h(\partial_{x^L},\partial_{x^K})}{h}-(-1)^{|\partial_{x^L}||\partial_{x^K}|}\frac{H_{M_1}^h(\partial_{x^K},\partial_{x^L})}{h}\nonumber\\
&=\partial_{x^L}(\partial_{x^K}(h))-\nabla_{\partial_{x^L}}^{L,M_1}\partial_{x^K}(h)-(-1)^{|\partial_{x^L}||\partial_{x^K}|}\partial_{x^K}(\partial_{x^L}(h))+(-1)^{|\partial_{x^L}||\partial_{x^K}|}\nabla_{\partial_{x^K}}^{L,M_1}\partial_{x^L}(h)\nonumber\\
&=-[-[\partial_{x^L},\partial_{x^K}](h)+\nabla_{\partial_{x^L}}^{L,M_1}\partial_{x^K}(h)-(-1)^{|\partial_{x^L}||\partial_{x^K}|}\partial_{x^K}(\partial_{x^L}(h))]\nonumber\\
&=-T^{L,M_1}(\partial_{x^L},\partial_{x^K})(h)\nonumber\\
&=0,
\end{align}
so (1) holds.\\
(2)By Definition \ref{def7}, we get
\begin{align}\label{f18}
{\rm Ric}^{L,\mu}(\partial_{x^L},\partial_{y^J})&=\sum_I(-1)^{|\partial_{x^I}|(|\partial_{x^I}|+|\partial_{x^L}|+|\partial_{y^J}|)}\frac{1}{2}[R^{L,\mu}(\partial_{x^I},\partial_{x^L})\partial_{y^J}+(-1)^{|\partial_{x^L}||\partial_{y^J}|}R^{L,\mu}(\partial_{x^I},\partial_{y^J})\partial_{x^L}]^I\nonumber\\
&+\sum_K(-1)^{|\partial_{y^K}|(|\partial_{y^K}|+|\partial_{x^L}|+|\partial_{y^J}|)}\frac{1}{2}[R^{L,\mu}(\partial_{y^K},\partial_{x^L})\partial_{x^K}+(-1)^{|\partial_{x^L}||\partial_{y^J}|}R^{L,\mu}(\partial_{y^K},\partial_{y^J})\partial_{x^L}]^K\nonumber\\
&=\sum_K(-1)^{|\partial_{y^K}|(|\partial_{y^K}|+|\partial_{x^L}|+|\partial_{y^J}|)}\frac{1}{2}[R^{L,\mu}(\partial_{y^K},\partial_{x^L})\partial_{x^K}+(-1)^{|\partial_{x^L}||\partial_{y^J}|}R^{L,\mu}(\partial_{y^K},\partial_{y^J})\partial_{x^L}]^K\nonumber\\
&=\sum_K(-1)^{|\partial_{y^K}|(|\partial_{y^K}|+|\partial_{x^L}|+|\partial_{y^J}|)}\frac{1}{2}[-(-1)^{|\partial_{y^K}||\partial_{x^L}|}R^{L,\mu}(\partial_{x^L},\partial_{y^K})\partial_{y^J}+(-1)^{|\partial_{x^L}||\partial_{y^J}|}\nonumber\\
&R^{L,\mu}(\partial_{y^K},\partial_{y^J})\partial_{x^L}]^K,
\end{align}
by Propsition \ref{prop2}, we get
\begin{align}\label{m90}
&R^{L,\mu}(\partial_{x^L},\partial_{y^K})\partial_{y^J}\nonumber\\
&=(-1)^{(|\partial_{y^K}|+|\partial_{x^L}|)|\partial_{y^J}|}\partial_{y^J}\left(\frac{\partial_{x^L}(h)}{h}\right)\partial_{y^K}-(-1)^{|\partial_{x^L}|(|\partial_{y^K}|+|\partial_{y^J}|)+|\partial_{y^K}||\partial_{y^J}|}g_2(\partial_{y^J},\partial_{y^K}){\rm grad}_{g_2}\frac{\partial_{x^L}(h)}{h}.
\end{align}
By \cite{BG}, we have \\
\begin{align}\label{f99}
{\rm grad}_gf=\sum^I(-1)^{|f||g|+|\partial_{y^J}|(|f|+|g|)}\frac{\partial f}{\partial_{y^J}}g^{JI}\partial_{y^I},\nonumber\\
\end{align}
and
\begin{align}\label{mkms}
g^{\alpha K}&=(-1)^{|\partial_{y^K}|^2+|\partial_{y^\alpha}|^2+|g|^2+|\partial_{y^K}||\partial_{y^\alpha}|}g^{K\alpha},
\end{align}
then
\begin{align}\label{kok}
&[(-1)^{|\partial_{x^L}|(|\partial_{y^K}|+|\partial_{y^J}|)+|\partial_{y^K}||\partial_{y^J}|}g_2(\partial_{y^J},\partial_{y^K}){\rm grad}_{g_2}(\partial_{x^L}(lnh))]^{\partial_{y^K}}\nonumber\\
&=(-1)^{|\partial_{x^L}|(|\partial_{y^K}|+|\partial_{y^J}|)+|\partial_{y^K}||\partial_{y^J}|}g_2(\partial_{y^J},\partial_{y^K})\sum_\alpha\partial_{x^L} (\partial_{y^\alpha}(lnh))g_2^{\alpha K}\nonumber\\
&=(-1)^{|\partial_{x^L}|(|\partial_{y^K}|+|\partial_{y^J}|)+|\partial_{y^K}||\partial_{y^J}|}(-1)^{(|\partial_{y^K}|+|\partial_{y^J}|)(|\partial_{y^\alpha}|+|\partial_{x^L}|)}\partial_{x^L} (\partial_{y^\alpha}(lnh))g_2(\partial_{y^J},\partial_{y^K})g_2^{\alpha K}\nonumber\\
&=\sum_\alpha(-1)^{|\partial_{y^\alpha}|+|\partial_{y^K}|+|\partial_{y^\alpha}||\partial_{y^K}|}(-1)^{(|\partial_{y^J}|+|\partial_{y^K}|)|\partial_{x^L}|+|\partial_{y^J}||\partial_{y^K}|}(-1)^{(|\partial_{y^J}|+|\partial_{y^K}|)(|\partial_{x^L}|+|\partial_{y^\alpha}|)}\nonumber\\
&\partial_{x^L} (\partial_{y^\alpha}(lnh))g_2(\partial_{y^J},\partial_{y^K})g_2^{K\alpha}\nonumber\\
&=\sum_\alpha(-1)^{|\partial_{y^K}|^2+|\partial_{y^\alpha}|}(-1)^{(|\partial_{y^\alpha}|+|\partial_{y^K}|)|\partial_{y^J}|}\partial_{x^L} (\partial_{y^\alpha}(lnh))g_2(\partial_{y^J},\partial_{y^K})g_2^{K\alpha},
\end{align}
so
\begin{align}\label{mklm}
&\sum_K(-1)^{(|\partial_{y^K}|+|\partial_{y^J}|+|\partial_{x^L}|)|\partial_{y^K}|}(-1)^{|\partial_{y^K}||\partial_{x^L}|}[(-1)^{|\partial_{x^L}|(|\partial_{y^K}|+|\partial_{y^J}|)+|\partial_{y^K}||\partial_{y^J}|}g_2(\partial_{y^J},\partial_{y^K}){\rm grad}_{g_2}(\partial_{x_L}(lnh))]^{\partial_{y^K}}\nonumber\\
&=\sum_{K,\alpha}(-1)^{|\partial_{y^\alpha}|+|\partial_{y^J}||\partial_{y^\alpha}|}\partial_{x^L}(\partial_{y^\alpha}(lnh))g_2(\partial_{y^J},\partial_{y^K})g_2^{K\alpha}\nonumber\\
&=\sum_{\alpha}(-1)^{|\partial_{y^\alpha}|+|\partial_{y^J}||\partial_{y^\alpha}|}\partial_{x^L}(\partial_{y^\alpha}(lnh))\delta_J^\alpha\nonumber\\
&=(-1)^{|\partial_{y^J}|+|\partial_{y^J}|^2}\partial_{x^L}(\partial_{y^J}(lnh))\nonumber\\
&=\partial_{x^L}(\partial_{y^J}(lnh))\nonumber\\
&=(-1)^{|\partial_{y^J}||\partial_{x^L}|}\partial_{y^J}(\partial_{x^L}(lnh))\nonumber\\
&=(-1)^{|\partial_{y^J}||\partial_{x^L}|}\partial_{y^J}\left(\frac{\partial_{x^L}(h)}{h}\right).
\end{align}
And
\begin{align}\label{xxx}
&\frac{1}{2}\sum_K(-1)^{|\partial_{y^K}|(|\partial_{y^K}|+|\partial_{x^L}|+|\partial_{y^J}|)}[-(-1)^{|\partial_{y^K}||\partial_{x^L}|}R^{L,\mu}(\partial_{x^L},\partial_{y^K})\partial_{y^J}]^{\partial_{y^K}}\nonumber\\
&=-\frac{1}{2}(q-m_2-1)(-1)^{|\partial_{x_L}||\partial_{y^J}|}\partial_{y^J}\left(\frac{\partial_{x^L}(h)}{h}\right),
\end{align}
similarly,
\begin{align}\label{ddd}
&\frac{1}{2}\sum_K(-1)^{|\partial_{y^K}|(|\partial_{y^K}|+|\partial_{x^L}|+|\partial_{y^J}|)}[(-1)^{|\partial_{x^L}||\partial_{y^J}|}R^{L,\mu}(\partial_{y^K},\partial_{y^J})\partial_{x^L}]^K\nonumber\\
&=-\frac{1}{2}(q-m_2-1)(-1)^{|\partial_{x_L}||\partial_{y^J}|}\partial_{y^J}\left(\frac{\partial_{x^L}(h)}{h}\right),
\end{align}
so (2) holds.\\
(3)By ${\rm Ric}^{L,\mu}(\partial_{y^J},\partial_{x^L})=(-1)^{|\partial_{x^L}||\partial_{y^J}|}{\rm Ric}^{L,\mu}(\partial_{x^L},\partial_{y^J}),$ (3) holds.\\
(4)By Definition \ref{def7} and Proposition \ref{prop2}, we have
\begin{align}\label{f19}
{\rm Ric}^{L,\mu}(\partial_{y^L},\partial_{y^J})&=\sum_I(-1)^{|\partial_{x^I}|(|\partial_{x^I}|+|\partial_{y^L}|+|\partial_{y^J}|)}\frac{1}{2}[R^{L,\mu}(\partial_{x^I},\partial_{y^L})\partial_{y^J}+(-1)^{|\partial_{y^L}||\partial_{y^J}|}R^{L,\mu}(\partial_{x^I},\partial_{y^J})\partial_{y^L}]^I\nonumber\\
&+\sum_K(-1)^{|\partial_{y^K}|(|\partial_{y^K}|+|\partial_{y^L}|+|\partial_{y^J}|)}\frac{1}{2}[R^{L,\mu}(\partial_{y^K},\partial_{y^L})\partial_{y^J}+(-1)^{|\partial_{y^L}||\partial_{y^J}|}R^{L,\mu}(\partial_{y^K},\partial_{y^J})\partial_{y^L}]^K\nonumber\\
&=\Delta_1+\Delta_2,
\end{align}
where
\begin{align}\label{1}
&\Delta_1:=\sum_I(-1)^{|\partial_{x^I}|(|\partial_{x^I}|+|\partial_{y^L}|+|\partial_{y^J}|)}\frac{1}{2}[R^{L,\mu}(\partial_{x^I},\partial_{y^L})\partial_{y^J}+(-1)^{|\partial_{y^L}||\partial_{y^J}|}R^{L,\mu}(\partial_{x^I},\partial_{y^J})\partial_{y^L}]^I\nonumber\\
&\Delta_2:=\sum_K(-1)^{|\partial_{y^K}|(|\partial_{y^K}|+|\partial_{y^L}|+|\partial_{y^J}|)}\frac{1}{2}[R^{L,\mu}(\partial_{y^K},\partial_{y^L})\partial_{y^J}+(-1)^{|\partial_{y^L}||\partial_{y^J}|}R^{L,\mu}(\partial_{y^K},\partial_{y^J})\partial_{y^L}]^K.
\end{align}
by Propsition \ref{b3}, we have
\begin{align}\label{f20}
&\sum_I(-1)^{|\partial_{x^I}|(|\partial_{x^I}|+|\partial_{y^L}|+|\partial_{y^J}|)}[R^{L,\mu}(\partial_{x^I},\partial_{y^L})\partial_{y^J}]^I\nonumber\\
&=-\sum_I(-1)^{|\partial_{x^I}|(|\partial_{x^I}|+|g|)}\frac{g_\mu(\partial_{y^L},\partial_{y^J})}{h}\bigg[\frac{\nabla^{L,M_1}_{\partial_{x^I}}({\rm grad}_{g_1}h)}{h}\bigg]^I\nonumber\\
&=-\frac{g_\mu(\partial_{y^L},\partial_{y^J})}{h}\sum_I(-1)^{|\partial_{x^I}|(|\partial_{x^I}|+|{\rm grad}_{g_1}h|)}\bigg[\frac{\nabla^{L,M_1}_{\partial_{x^I}}({\rm grad}_{g_1}h)}{h}\bigg]^I\nonumber\\
&=-\frac{g_\mu(\partial_{y^L},\partial_{y^J})}{h}Div_{\nabla^{L,M_1}}{\rm grad}_{g_1}h\nonumber\\
&=-\frac{g_\mu(\partial_{y^L},\partial_{y^J})}{h}\triangle_{g_1}^L(h),
\end{align}
then, we get
\begin{align}\label{f21}
\Delta_1&=\frac{1}{2}[-g_\mu(\partial_{y^L},\partial_{y^J})\frac{\triangle^L_{g_1}(h)}{h}-(-1)^{|\partial_{y^L}||\partial_{y^J}|}g_\mu(\partial_{y^L},\partial_{y^J})\frac{\triangle^L_{g_1}(h)}{h}],\nonumber\\
&=-\frac{g_\mu(\partial_{y^L},\partial_{y^J})}{h}\triangle_{g_1}^L(h),
\end{align}
\begin{align}\label{f22}
\Delta_2&={\rm Ric}^{L,M_2}(\partial_{y^L},\partial_{y^J})+\sum_K(-1)^{|\partial_{y^K}|(|\partial_{y^K}|+|\partial_{y^L}|+|\partial_{y^J}|)}\frac{1}{2}\bigg\{(-1)^{|\partial_{y^L}||\partial_{y^J}|}g_\mu(\partial_{y^K},\partial_{y^J}){\rm grad}_{g_2}\frac{\partial_{y^L}(h)}{h}\nonumber\\
&-(-1)^{(|\partial_{y^L}|+|\partial_{y^J}|)|\partial_{y^K}|}g_\mu(\partial_{y^L},\partial_{y^J}){\rm grad}_{g_2}\frac{\partial_{y^K}(h)}{h}-(-1)^{(|\partial_{y^L}|+|\partial_{y^J}|)|\partial_{y^K}|}\nonumber\\
&-(-1)^{|\partial_{y^K}|(|\partial_{y^L}|+|\partial_{y^J}|)}\frac{({\rm grad}_{g_1}h)(h)}{h^2}g_\mu(\partial_{y^L},\partial_{y^J})\partial_{y^K}+(-1)^{|\partial_{y^L}||\partial_{y^J}|}\frac{({\rm grad}_{g_1}h)(h)}{h^2}g_\mu(\partial_{y^K},\partial_{y^J})\partial_{y^L}\nonumber\\
&+(-1)^{|\partial_{y^L}||\partial_{y^J}|}\bigg[(-1)^{|\partial_{y^L}||\partial_{y^J}|}g_\mu(\partial_{y^K},\partial_{y^L}){\rm grad}_{g_2}\frac{\partial_{y^J}(h)}{h}\nonumber\\
&-(-1)^{(|\partial_{y^L}|+|\partial_{y^J}|)|\partial_{y^K}|}g_\mu(\partial_{y^K},\partial_{y^L})({\rm grad}_{g_2}\frac{\partial_{y^K}(h)}{h}-(-1)^{(|\partial_{y^L}|+|\partial_{y^J}|)|\partial_{y^K}|}\frac{({\rm grad}_{g_1}h)(h)}{h^2}g_\mu(\partial_{y^J},\partial_{y^L})\partial_{y^K}\nonumber\\
&+(-1)^{|\partial_{y^L}||\partial_{y^J}|}\frac{({\rm grad}_{g_1}h)(h)}{h^2}g_\mu(\partial_{y^K},\partial_{y^L})\partial_{y^J}\bigg]\bigg\},\nonumber\\
&=Ric^{L,M_2}(\partial_{y^L},\partial_{y^J})-g_\mu(\partial_{y^L},\partial_{y^J})(q-m_2-1)\frac{({\rm grad}_{g_1}h)(h)}{h^2},
\end{align}
then,\\
\begin{align}\label{f23}
&\Delta_1+\Delta_2={\rm Ric}^{L,M_2}(\partial_{y^L},\partial_{y^J})-g_\mu(\partial_{y^L},\partial_{y^J})
[\frac{\triangle^L_{g_1}(h)}{h}+(q-m_2-1)\frac{({\rm grad}_{g_1}h)(h)}{h^2}],
\end{align}
so (4) holds.\\
\end{proof}
\begin{thm}\label{thm:9}
Let $M=M_1\times M_2$ be a singly twisted product manifold with the
metric tensor $g=g_1\oplus h^2g_2$. If $X,Y,Z\in Vect(M_1)$, $U,V,Q\in Vect(M_2)$, then
\begin{align}\label{nnn}
(1)K^{L,\mu}(X,Y)Z&=K^{L,M_1}(X,Y)Z+\frac{n_2}
{(m-n-1)(n_1-1)}[X\cdot Ric^{L,M_1}(Y,Z)-(-1)^{|Y||Z|}Ric^{L,M_1}(X,Z)Y]\nonumber\\
&+\frac{q-m_2}{(m-n-1)h}[X\cdot H_{M_1}^h(Y,Z)-(-1)^{|Y||Z|}H_{M_1}^h(X,Z)Y]\nonumber\\
(2)K^{L,\mu}(X,Y)Q&=-\frac{q-m_2-1}{m-n-1}\left[(-1)^{|Y||Q|}Q\left(\frac{X(h)}{h}\right)Y-X\cdot Q\left(\frac{Y(h)}{h}\right)\right]\nonumber\\
(3)K^{L,\mu}(U,V)X&=(-1)^{|V||X|}U\left(\frac{X(h)}{h}\right)V+(-1)^{|U|(|V|+|X|)}V\left(\frac{X(h)}{h}\right)U\nonumber\\
&+\frac{q-m_2-1}{m-n-1}\left[(-1)^{|V||X|}U\cdot X\left(\frac{X(h)}{h}\right)-(-1)^{|X|(|V|+|U|)}X\left(\frac{U(h)}{h}\right)V\right]\nonumber\\
(4)K^{L,\mu}(X,V)Y&=\frac{1}{m-n-1}(-1)^{|V||Y|}[(m-n-q+m_2-1)H_{M_1}^h(X,Y)+Ric^{L,M_1}(X,Y)]V\nonumber\\
&+\frac{q-m_2-1}{m-n-1}X\cdot Y\frac{V(h)}{h}\nonumber\\
(5)K^{L,\mu}(X,U)V&=(-1)^{|X|(|V|+|U|)}V\left(\frac{X(h)}{h}\right)U-(-1)^{|X|(|V|+|U|)+|U||V|}g_2(U,V)[h\triangle_{g_1}^L(h)\nonumber\\
&+(q-m_2-1){\rm grad}_{g_1}(h)(h)]+\frac{q-m_2-1}{m-n-1}(-1)^{|V||U|}V\frac{X(h)}{h}U\nonumber\\
(6)K^{L,\mu}(U,V)Q&=K^{L,M_2}(U,V)Q+\frac{n_1}
{(m-n-1)(n_2-1)}[U\cdot Ric^{L,M_2}(V,Q)-(-1)^{|V||Q|}Ric^{L,M_2}(U,Q)V]\nonumber\\
&+(-1)^{|V||Q|}g_\mu(U,Q){\rm grad}_{g_2}\frac{V(h)}{h}-(-1)^{|U|(|V|+|Q|)}g_\mu(V,Q){\rm grad}_{g_2}\frac{U(h)}{h}\nonumber\\
&-(-1)^{|U|(|V|+|Q|)}\frac{{\rm grad}_{g_1}(h)(h)}{h^2}g_2(V,Q)U+(-1)^{|V||Q|}\frac{{\rm grad}_{g_1}(h)(h)}{h^2}g_2(U,Q)V\nonumber\\
&+\frac{1}{m-n-1}[U\cdot g_2(V,Q)(h\triangle_{g_1}^L(h)+(q-m_2-1){\rm grad}_{g_1}(h)(h))]\nonumber\\
&-\frac{1}{m-n-1}[(-1)^{|V||Q|} g_2(U,Q)(h\triangle_{g_1}^L(h)+(q-m_2-1){\rm grad}_{g_1}(h)(h))V].
\end{align}
\end{thm}
\begin{proof}
(1)By Definition \ref{def1p} and Propsition \ref{prop12}, we have
\begin{align}\label{m1}
K^{L,\mu}(X,Y)Z&=R^{L,\mu}(X,Y)Z-\frac{1}{(m-n-1)}[X\cdot Ric^{L,\mu}(Y,Z)-(-1)^{|Y||Z|}Ric^{L,\mu}(X,Z)Y]\nonumber\\
&=R^{L,M_1}(X,Y)Z-\frac{1}{(m-n-1)}X\cdot[Ric^{L,M_1}(Y,Z)-\frac{q-m_2}{h}H_{M_1}^h(Y,Z)]\nonumber\\
&+\frac{1}{(m-n-1)}(-1)^{|Y||Z|}[Ric^{L,M_1}(X,Z)-\frac{q-m_2}{h}H_{M_1}^h(X,Z)]Y\nonumber\\
&=K^{L,M_1}(X,Y)Z+\frac{n_2}
{(m-n-1)(n_1-1)}[X\cdot Ric^{L,M_1}(Y,Z)-(-1)^{|Y||Z|}Ric^{L,M_1}(X,Z)Y]\nonumber\\
&+\frac{q-m_2}{(m-n-1)h}[X\cdot H_{M_1}^h(Y,Z)-(-1)^{|Y||Z|}H_{M_1}^h(X,Z)Y],
\end{align}
then (1) holds.\\
(2)By Definition \ref{def1p} and Propsition \ref{prop12}, we have
\begin{align}\label{m2}
K^{L,\mu}(X,Y)Q&=R^{L,\mu}(X,Y)Q-\frac{1}{(m-n-1)}[X\cdot Ric^{L,\mu}(Y,Q)-(-1)^{|Y||Q|}Ric^{L,\mu}(X,Q)Y]\nonumber\\
&=-\frac{1}{(m-n-1)}\left[X\cdot-\frac{q-m_2-1}{2}Q\left(\frac{Y(h)}{h}\right)+(-1)^{|Y||Q|}\frac{q-m_2-1}{2}Q\left(\frac{X(h)}{h}\right)Y\right]\nonumber\\
&=-\frac{q-m_2-1}{2(m-n-1)}\left[(-1)^{|Y||Z|}Q\left(\frac{X(h)}{h}\right)Y-X\cdot Q\left(\frac{Y(h)}{h}\right)\right],
\end{align}
so (2) holds.\\
(3)By Definition \ref{def1p} and Propsition \ref{prop12}, we have
\begin{align}\label{m3}
K^{L,\mu}(U,V)X&=R^{L,\mu}(U,V)X-\frac{1}{(m-n-1)}[U\cdot Ric^{L,\mu}(V,X)-(-1)^{|V||X|}Ric^{L,\mu}(U,X)V]\nonumber\\
&-\frac{1}{(m-n-1)}\left[U\cdot -(q-m_2-1)X\left(\frac{V(h)}{h}\right)+(-1)^{(|V||X|}(q-m_2-1)X\left(\frac{U(h)}{h}\right)V\right]\nonumber\\
&=(-1)^{|V||X|}U\left(\frac{X(h)}{h}\right)V+(-1)^{|U|(|V|+|X|)}V\left(\frac{X(h)}{h}\right)U\nonumber\\
&-\frac{q-m_2-1}{2(m-n-1)h}\left[U\cdot X\left(\frac{V(h)}{h}\right)-(-1)^{(|V||X|}X\left(\frac{U(h)}{h}\right)V\right],
\end{align}
so (3) holds.\\
(4)By Definition \ref{def1p} and Propsition \ref{prop12}, we have
\begin{align}\label{m4}
K^{L,\mu}(X,V)Y&=R^{L,\mu}(X,V)Y-\frac{1}{(m-n-1)}[X\cdot Ric^{L,\mu}(V,Y)-(-1)^{|V||Y|}Ric^{L,\mu}(X,Y)V]\nonumber\\
&=(-1)^{|X||V|}R^{L,\mu}(V,X)Y-\frac{1}{m-n-1}[X\cdot -(q-m_2-1)Y\left(\frac{V(h)}{h}\right)\nonumber\\
&-(-1)^{(|V||Y|}(Ric^{L,M_1}(X,Y)-\frac{q-m_2}{h}H_{M_1}^h(X,Y))V]\nonumber\\
&=\frac{1}{m-n-1}(-1)^{|V||Y|}[(n-q+m_2-1)H_{M_1}^h(X,Y)+Ric^{L,M_1}(X,Y)]V\nonumber\\
&+\frac{q-m_2-1}{m-n-1}X\cdot Y\left(\frac{V(h)}{h}\right),
\end{align}
so (4) holds.\\
(5)By Definition \ref{def1p} and Propsition \ref{prop12}, we have
\begin{align}\label{m5}
K^{L,\mu}(X,U)V&=R^{L,\mu}(X,U)V-\frac{1}{m-n-1}[X\cdot Ric^{L,\mu}(U,V)-(-1)^{|U||V|}Ric^{L,\mu}(X,V)U]\nonumber\\
&=(-1)^{|X|(|U|+|V|)}V\left(\frac{X(h)}{h}\right)U-(-1)^{|X|(|U|+|V|)+|V||U|}g_2(V,U){\rm grad}_{g_2}\frac{X(h)}{h}\nonumber\\
&-(-1)^{|X|(|U|+|V|)}hg_2(U,V)\nabla^{L,M_1}_X({\rm grad}_{g_2}h)-\frac{1}{m-n-1}\bigg\{X\cdot\bigg[Ric^{L,M_2}(U,V)\nonumber\\
&-g_\mu(U,V)\bigg(\frac{\triangle^L_{g_1}(h)}{h}+(q-m_2-1)\frac{({\rm grad}_{g_1}h)(h)}{h^2}\bigg)\bigg]+(-1)^{|U||V|}\frac{(q-m_2-1)}{2}V\left(\frac{X(h)}{h}\right)U\bigg\}\nonumber\\
&=(-1)^{|X|(|U|+|V|)}V\left(\frac{X(h)}{h}\right)U-(-1)^{|X|(|U|+|V|)+|V||U|}g_2(V,U){\rm grad}_{g_2}\frac{X(h)}{h}\nonumber\\
&-(-1)^{|X|(|U|+|V|)}hg_2(U,V)\nabla^{L,M_1}_X({\rm grad}_{g_2}h)-\frac{1}{m-n-1}X\cdot Ric^{L,M_2}(U,V)\nonumber\\
&+\frac{1}{m-n-1}X\cdot g_2(U,V)[h\triangle^L_{g_1}(h)+(-1)^{|U||V|}\frac{(q-m_2-1)}{2}V\left(\frac{X(h)}{h}\right)U,
\end{align}
so (5) holds.\\
(6)By Definition \ref{def1p} and Propsition \ref{prop12}, we have
\begin{align}\label{m6}
K^{L,\mu}(U,V)Q&=R^{L,\mu}(U,V)Q-\frac{1}{m-n-1}[U\cdot Ric^{L,\mu}(V,Q)-(-1)^{|V||Q|}Ric^{L,\mu}(U,Q)V]\nonumber\\
&=R^{L,M_2}(U,V)Q+(-1)^{|Q||V|)}g_\mu(U,Q){\rm grad}_{g_2}\frac{V(h)}{h}\nonumber\\
&-(-1)^{|U|(|Q|+|V|)}g_\mu(V,Q){\rm grad}_{g_2}\frac{U(h)}{h}-(-1)^{|U|(|Q|+|V|)}\frac{({\rm grad}_{g_1}h)(h)}{h^2}g_2(V,Q)U\nonumber\\
&+(-1)^{|V||Q|}\frac{({\rm grad}_{g_1}h)(h)}{h^2}g_2(U,Q)V\nonumber\\
&-\frac{1}{m-n-1}\bigg\{U\cdot \bigg[Ric^{L,M_2}(V,Q)-g_\mu(V,Q)\bigg(\frac{\triangle^L_{g_1}(h)}{h}+(q-m_2-1)\frac{({\rm grad}_{g_1}h)(h)}{h^2}\bigg)\bigg]\nonumber\\
&-(-1)^{|V||Q|}\bigg[Ric^{L,M_2}(U,Q)-g_\mu(U,Q)\bigg(\frac{\triangle^L_{g_1}(h)}{h}+(q-m_2-1)\frac{({\rm grad}_{g_1}h)(h)}{h^2}\bigg)\bigg]V\bigg\}\nonumber\\
&=K^{L,M_2}(U,V)Q+\frac{n_1}
{(m-n-1)(n_2-1)}[U\cdot Ric^{L,M_2}(V,Q)-(-1)^{|V||Q|}Ric^{L,M_2}(U,Q)V]\nonumber\\
&+(-1)^{|Q||V|)}g_\mu(U,Q){\rm grad}_{g_2}\frac{V(h)}{h}-(-1)^{|U|(|Q|+|V|)}g_\mu(V,Q){\rm grad}_{g_2}\frac{U(h)}{h}\nonumber\\
&-(-1)^{|U|(|Q|+|V|)}\frac{({\rm grad}_{g_1}h)(h)}{h^2}g_2(V,Q)U+(-1)^{|V||Q|}\frac{({\rm grad}_{g_1}h)(h)}{h^2}g_2(U,Q)V\nonumber\\
&+\frac{1}{m-n-1}U\cdot \bigg[g_2(V,Q)(h\triangle^L_{g_1}(h))+(q-m_2-1)({\rm grad}_{g_1}h)(h)\bigg]\nonumber\\
&-\frac{1}{m-n-1}\bigg[(-1)^{|V||Q|}g_2(U,Q)(h\triangle^L_{g_1}(h))+(q-m_2-1)({\rm grad}_{g_1}h)(h)\bigg]V,
\end{align}
so (6) holds.\\
\end{proof}
\section{Mixed Ricci flat super twisted products}
\label{Section:4}
\begin{defn}\label{defF1}Let $M=M_1\times_\mu M_2$ be a super twisted product of $(M_1,g_1)$ and $(M_2,g_2)$ with twisting
function $h$, then $M=M_1\times_\mu M_2$ is called mixed Ricci-flat if $Ric(X,V)=0$ for all
$X\in  Vect(M_1)$ and $V\in  Vect(M_2)$.
\end{defn}
\begin{thm}\label{thm:5}
Let $M=M_1\times_\mu M_2$ be a super twisted product of $(M_1,g_1)$ and $(M_2,g_2)$ with twisting
function $h$ and $q-m_2-1\neq0$. Then, $Ric(X,V)=0$ for all
$X\in  Vect(M_1)$ and $V\in  Vect(M_2)$
if and only if $M=M_1\times_\mu M_2$ can be expressed as a super warped product, $M=M_1\times_\mu M_2$ of $(M_1,g_1)$ and
$(M_2,\widehat{g_2})$ with a warping function $\widehat{\Phi}$, where $\widehat{g_2}$ is a conformal metric tensor to $g_2$.
\end{thm}
\begin{proof}
First, we prove the sufficiency of the theorem. By Propsiton \ref{prop12} and $\frac{X(h)}{h}=X(lnh)$, we know
\begin{align}\label{kkk}
{\rm Ric}^{L,\mu}(X,V)&=-(q-m_2-1)(-1)^{|X||V|}VX(lnh)=0,
\end{align}
then by $q-m_2-1\neq0$, $VX(lnh)=0$ and $XV(lnh)=0$,  $XV(lnh)=0$ implies that $V(lnh)$ only depends
on the points of $M_2$, and similarly, $VX(lnh)=0$ implies that $X(lnh)$ only depends
on the points of $M_1$. Thus $h$ can be expressed as a sum of two functions $\Phi$ and $\Psi$
which are defined on $M_1$ and $M_2$, respectively, that is, $lnh(s,t)=\phi(s)+\psi(t)$ for
any $(s,t)\in M_1\times M_2$. Hence $h=e^\phi e^\psi$, that is, $h=\Phi(s)\Psi(t)$,
where $\Phi=e^\phi$ and $\Psi=e^\psi$ for any $(s,t)\in M_1\times M_2$. Thus we can write
$g=g_1\oplus \Phi^2\widehat{g_2}$, where $\widehat{g_2}= \Psi^2g_2$ , that is, a super twisted product $M_1\times_\mu M_2$ can be
expressed as a super warped product $M_1\times_\mu M_2$, where the metric tensor of $M_2$ is $\widehat{g_2}$ given
above.\\
By Proposition \ref{prop12}, we find that it's obvious about the necessity.
\end{proof}
\begin{thm}\label{thm:6}
Let $M=M_1\times_\mu M_2$ be a super twisted product of $(M_1,g_1)$ and $(M_2,g_2)$ with twisting
function $h$. If $M$ is a $W_2$-curvature flat super twisted product,
then $M=M_1\times_\mu M_2$ can be expressed as a super warped product.
\end{thm}
\begin{proof}
By Theorem \ref{thm:9}, we know
\begin{align}\label{nbn}
K^{L,\mu}(X,Y)Q&=-\frac{q-m_2-1}{m-n-1}\left[(-1)^{|Y||Q|}Q\frac{X(h)}{h}Y-X\cdot Q\frac{Y(h)}{h}\right]=0.
\end{align}
If $q-m_2-1\neq0,$ let $Q=\partial_{y^K},~X=\partial_{x^I},~Y=\partial_{x^J},$ when $I\neq J,$ we have\\
\begin{align}\label{kjk}
(-1)^{|\partial_{x^J}||\partial_{y^K}|}\partial_{y^K}[\partial_{x^I}(lnh)]\partial_{x^J}-\partial_{x^I}\cdot\partial_{y^K}[\partial_{x^J}(lnh)]=0,
\end{align}
Because a pair $(I,J)$ is arbitrary, then $I\neq J$ implies that $\partial_{y^K}[\partial_{x^I}(lnh)]=0$.
So
\begin{align}\label{opop}
lnh=\phi(x)+\psi(y),
\end{align}
for any $(x,y)\in M_1\times M_2$, then $h=e^{\phi(x)}e^{\psi(y)}.$\\
If $q-m_2-1=0$, then by Theorem \ref{thm:9}, we know
\begin{align}
K^{L,\mu}(U,V)X=(-1)^{|V||X|}UX(lnh)V+(-1)^{|U|(|V|+|X|)}VX(lnh)U=0,
\end{align}
similarly, let $U=\partial_{y^P},~V=\partial_{y^Q},~X=\partial_{x^I},$ when $P\neq Q,$ we have
\begin{align}\label{kjk}
(-1)^{|\partial_{x^I}||\partial_{y^Q}|}\partial_{y^P}[\partial_{x^I}(lnh)]\partial_{y^Q}+(-1)^{|\partial_{y^P}|(|\partial_{y^Q}|+|\partial_{x^I}|)}\partial_{y^Q}[\partial_{x^I}(lnh)]\partial_{y^P}=0.
\end{align}
Similar to (\ref{opop}), we can get
\begin{align}\label{oporp}
lnh=\omega(x)+\nu(y),
\end{align}
for any $(x,y)\in M_1\times M_2$, then $h=e^{\omega(x)}e^{\nu(y)},$ therefore we can get Theorem \ref{thm:6}.
\end{proof}
\begin{cor}\label{zxz}
Let $M=M_1\times_\mu M_2$ be a super twisted product of $(M_1,g_1)$ and $(M_2,g_2)$ with twisting
function $h$. Then, $M$ is a super warped product if and only if $K^{L,\mu}(X,Y)Q=0$ for $X,Y \in Vect(M_1), Q\in Vect(M_2)$.
\end{cor}

\section*{Acknowledgements}
The work was supported in part by  NSFC (No.11771070). The authors thank the referee for his (or her) careful reading and helpful comments.

\section*{}

\end{document}